\documentclass[a4paper,oneside,11pt] {amsart}
\usepackage{amsfonts,amssymb,amsmath,amsthm}
\usepackage{url}
\usepackage{enumerate}
\usepackage{mathrsfs}

\usepackage{subfigure,caption,float}
\usepackage{graphicx}

\newcommand{\ncom}{\newcommand}
\ncom{\ul}{\underline}
\ncom{\ol}{\overline}
\ncom{\bq}{\begin{equation}}
\ncom{\eq}{\end{equation}}
\ncom{\beqn}{\begin{eqnarray*}}
\ncom{\eeqn}{\end{eqnarray*}}
\ncom{\beq}{\begin{eqnarray}}
\ncom{\eeq}{\end{eqnarray}}
\ncom{\nno}{\nonumber}
\ncom{\rar}{\rightarrow}
\ncom{\Rar}{\Rightarrow}
\ncom{\noin}{\noindent}
\ncom{\bc}{\begin{centre}}
\ncom{\ec}{\end{centre}}
\ncom{\sz}{\scriptsize}
\ncom{\rf}{\ref}
\ncom{\sgm}{\sigma}
\ncom{\Sgm}{\Sigma}
\ncom{\dt}{\delta}
\ncom{\Dt}{Delta}
\ncom{\lmd}{\lambda}
\ncom{\Lmd}{\Lambda}
\ncom{\eps}{\epsilon}
\ncom{\pcc}{\stackrel{P}{>}}
\ncom{\dist}{{\rm\,dist}}
\ncom{\sspan}{{\rm\,span}}
\ncom{\re}{{\rm Re\,}}
\ncom{\im}{{\rm Im\,}}
\ncom{\sgn}{{\rm sgn\,}}
\ncom{\ba}{\begin{array}}
\ncom{\ea}{\end{array}}
\ncom{\eop}{\hfill{{\rule{2.5mm}{2.5mm}}}}
\ncom{\eoe}{\hfill{{\rule{1.5mm}{1.5mm}}}}
\ncom{\eof}{\hfill{{\rule{1.5mm}{1.5mm}}}}
\ncom{\hone}{\mbox{\hspace{1em}}}
\ncom{\htwo}{\mbox{\hspace{2em}}}
\ncom{\hthree}{\mbox{\hspace{3em}}}
\ncom{\hfour}{\mbox{\hspace{4em}}}
\ncom{\hsev}{\mbox{\hspace{7em}}}
\ncom{\vone}{\vskip 2ex}
\ncom{\vtwo}{\vskip 4ex}
\ncom{\vonee}{\vskip 1.5ex}
\ncom{\vthree}{\vskip 6ex}
\ncom{\vfour}{\vspace*{8ex}}
\ncom{\norm}{\|\;\;\|}
\ncom{\integ}[4]{\int_{#1}^{#2}\,{#3}\,d{#4}}
\ncom{\inp}[2]{\langle{#1},\,{#2} \rangle}
\ncom{\Inp}[2]{\Langle{#1},\,{#2} \Langle}
\ncom{\vspan}[1]{{{\rm\,span}\#1 \}}}
\ncom{\dm}[1]{\displaystyle {#1}}

\newtheorem{theorem}{\bf Theorem}[section]
\newtheorem{corollary}[theorem]{\bf Corollary}
\newtheorem{lemma}[theorem]{\bf Lemma}

\newtheorem{question}[theorem]{\bf Question}

\newtheoremstyle
    {remarkstyle}
    {}
    {11pt}
    {}
    {}
    {\bfseries}
    {:}
    {     }
    {\thmname{#1} \thmnumber{#2} }

\theoremstyle{remarkstyle}

\newtheorem{remark}[theorem]{\bf Remark}
\newtheorem{example}[theorem]{\bf Example}

\title[Module Tensor Product of Subnormal Modules]
{Module Tensor Product of Subnormal Modules need not be Subnormal}
\author[A. Anand and S. Chavan]{Akash Anand and
Sameer Chavan}
\address{Indian Institute of Technology Kanpur\\
Kanpur- 208016, India}
\email{akasha@iitk.ac.in}
\email{chavan@iitk.ac.in}



\begin{document}

\begin{abstract}
Let $\kappa : \mathbb D \times \mathbb D \rar \mathbb C$ be a diagonal positive definite kernel and let $\mathscr H_{\kappa}$ denote the associated reproducing kernel Hilbert space of holomorphic functions on the open unit disc $\mathbb D$. Assume that $zf \in \mathscr H$ whenever $f \in \mathscr H.$ Then $\mathscr H$ is a Hilbert module over the polynomial ring $\mathbb C[z]$ with module action
$p \cdot f \mapsto pf$. We say that $\mathscr H_{\kappa}$ is a subnormal Hilbert module if 
the operator $\mathscr M_{z}$ of multiplication by the coordinate function $z$ on $\mathscr H_{\kappa}$ is subnormal.
In [Oper. Theory Adv. Appl, 32: 219-241, 1988], N. Salinas asked whether the module tensor product 
$\mathscr H_{\kappa_1} \otimes_{\mathbb C[z]} \mathscr H_{\kappa_2}$
of subnormal Hilbert modules $\mathscr H_{\kappa_1}$ and $\mathscr H_{\kappa_2}$ is again subnormal.
In this regard, we describe all subnormal module tensor products 
$L^2_a(\mathbb D, w_{s_1}) \otimes_{\mathbb C[z]} L^2_a(\mathbb D, w_{s_2})$, where $L^2_a(\mathbb D, w_s)$ denotes the weighted Bergman Hilbert module with radial weight $$w_s(z)=\frac{1}{s \pi}|z|^{\frac{2(1-s)}{s}}~(z \in \mathbb D, ~s > 0).$$
In particular, the module tensor product 
$L^2_a(\mathbb D, w_{s}) \otimes_{\mathbb C[z]} L^2_a(\mathbb D, w_{s})$ is never subnormal for any $s \geq 6$.
Thus the answer to this question is no.
\end{abstract}

\maketitle


\section{Introduction}

Let $\mathscr H$ be a reproducing kernel Hilbert space of holomorphic functions defined on the unit disc $\mathbb D$ such that $zf \in \mathscr H$ whenever $f \in \mathscr H.$
Thus the linear operator $\mathscr M_z$ of multiplication by the coordinate function $z$ on $\mathscr H$ is bounded. This allows us to realize $\mathscr H$ as a {\it Hilbert module} over the polynomial ring $\mathbb C[z]$ with
module action given by 
\beqn
(p, f) \in \mathbb C[z] \times \mathscr H \longmapsto p(\mathscr M_z)f \in \mathscr H.
\eeqn
Following \cite{D}, we say that the Hilbert module $\mathscr H$ is 
{\it contractive} if the operator norm of the multiplication operator $\mathscr M_z$ is at most $1.$ Further, we say that $\mathscr H$ is 
{\it subnormal} if $\mathscr M_z$ is subnormal, that is, $\mathscr M_z$ has a normal extension in a Hilbert module containing $\mathscr H$ (refer to \cite{Co} for a comprehensive account on subnormal operators).
By Agler's Criterion \cite[Theorem 3.1]{Ag},
$\mathscr H$ is a contractive subnormal Hilbert module if and only if for every $f \in \mathscr H,$ 
$\phi_f(n)=\|z^nf\|^2~(n \in \mathbb N)$ is completely monotone for every $f \in \mathscr H.$ 
Recall from \cite{BCR} that $\phi : \mathbb N \rar (0, \infty)$ is {\it completely monotone} if $$\sum_{j=0}^m(-1)^j {m \choose j}\phi(n+j) \geq 0~\mbox{for all}~m, n \in \mathbb N.$$  
\begin{remark} \label{translate}
Note that if $\phi$ is a completely monotone sequence then so is $\psi_m$ for any $m \in \mathbb N$, where $\psi_m(n)=\phi(m +n)~(n \in \mathbb N)$. 
\end{remark}
We further note that, as a consequence of Hausdorff's solution to the Hausdorff's moment problem \cite[Chapter 4, Proposition 6.11]{BCR}, $\mathscr H$ is a contractive subnormal Hilbert module if and only if for every unit vector $f \in \mathscr H,$ $\{\|z^nf\|^2\}_{n \in \mathbb N}$ 
is a {\it Hausdorff moment sequence}, that is, there exists a unique probability measure $\mu_f$ supported in $[0, 1]$ such that
$$\|z^nf\|^2=\int_{[0, 1]}t^n d\mu_f~(n \in \mathbb N).$$

In this text, we are primarily interested in the following one parameter family of subnormal Hilbert modules.
\begin{example}
For a real number $s > 0,$ consider the Hilbert space $L^2_a(\mathbb D, w_s)$ of holomorphic functions defined on the open unit disc $\mathbb D$ which are square integrable with respect to the weighted area measure $w_s\,dA$
 with radial weight function $$w_s(z)=\frac{1}{s \pi}|z|^{\frac{2(1-s)}{s}}~(z \in \mathbb D).$$
Then $L^2_a(\mathbb D, w_s)$ is a Hilbert module over the polynomial ring $\mathbb C[z]$ (refer to \cite{HKZ} for the basic theory of weighted Bergman spaces). 
Since $L^2_a(\mathbb D, w_s)$ is a closed subspace of $L^2(\mathbb D, w_s)$, $L^2_a(\mathbb D, w_s)$ is a subnormal Hilbert module.
Further, since the measure $w_s\,dA$ is rotation-invariant, the monomials $\{z^n\}_{n \in \mathbb N}$ are orthogonal in
$L^2_a(\mathbb D, w_s)$. Further, 
\beq \label{moments} 
\|z^n\|^2_{L^2_a(\mathbb D, w_s)} = \frac{1}{sn +1}~(n \in \mathbb N).
\eeq
In particular, $L^2_a(\mathbb D, w_s)$ is a reproducing kernel Hilbert space associated with the diagonal reproducing kernel 
\beqn
\kappa_s(z, w)=\frac{s}{(1-z\overline{w})^2} + \frac{1-s}{1-z\overline{w}}~(z, w \in \mathbb D).
\eeqn
\end{example}
\begin{remark}
Note that $L^2_a(\mathbb D, w_1)$ is the unweighted Bergman space $L^2_a(\mathbb D)$. It is worth noting that the definition of $L^2_a(\mathbb D, w_s), s >0$ extends to the case $s=0.$ Indeed, $L^2_a(\mathbb D, w_0)$ may be identified with the Hardy space $H^2(\mathbb D)$ of the open unit disc $\mathbb D$.
\end{remark}

In \cite{Sa}, N. Salinas studied the notion of module tensor product of Hilbert $\mathscr A$-modules for a complex unital algebra $\mathscr A$. Recall that for Hilbert $\mathscr A$-modules $\mathscr H$ and $\mathscr K$, $\mathscr H \otimes_{\mathscr A} \mathscr K$ is obtained by tensoring Hilbert modules $\mathscr H$ and $\mathscr K$, and then dividing out by the natural action of $\mathscr A$ on $\mathscr H$ and $\mathscr K$. In particular, in \cite[Corollary 3.6]{Sa}, it is shown that
for the polynomial algebra $\mathbb C[z]$ and for Hilbert modules $\mathscr H_{\kappa_1}$ and $\mathscr H_{\kappa_2}$ associated with so-called sharp, diagonal kernels $\kappa_1$ and $\kappa_2$ (of finite rank) respectively, the module tensor product $\mathscr H_{\kappa_1} \otimes_{\mathbb C[z]} \mathscr H_{\kappa_2}$ is $\mathbb C[z]$-isomorphic to the Hilbert space $\mathscr H_{\kappa_1 \kappa_2}$ associated with  the diagonal kernel $\kappa_1 \kappa_2.$
In case of scalar valued diagonal kernels $\kappa_1$ and $\kappa_2$, the moment sequences $\{\|z^n\|^2_{\mathscr H_{\kappa_1}}\}_{n \in \mathbb N}$, $\{\|z^n\|^2_{\mathscr H_{\kappa_2}}\}_{n \in \mathbb N}$, 
$\{\|z^n\|^2_{\mathscr H_{\kappa_1\kappa_2}}\}_{n \in \mathbb N}$ are related by the following relations:
\beq
\label{relation-m-s}
\|z^n\|^2_{\mathscr H_{\kappa_1\kappa_2}} =  \displaystyle \frac{1}{\displaystyle \sum_{k=0}^n \frac{1}{\|z^k\|^2_{\mathscr H_{\kappa_1}}} \frac{1}{\|z^{n-k}\|^2_{\mathscr H_{\kappa_2}}}}~(n \in \mathbb N).
\eeq
In \cite[Remark 3.7]{Sa}, he asked whether  $\mathscr H_{\kappa_1} \otimes_{\mathbb C[z]} \mathscr H_{\kappa_2}$ is a subnormal Hilbert module for subnormal Hilbert modules  $\mathscr H_{\kappa_1}$ and $\mathscr H_{\kappa_2}$ associated with diagonal scalar-valued kernels $\kappa_1$ and $\kappa_2$ respectively ? In view of \eqref{relation-m-s} and the discussion following Remark \eqref{translate}, this is equivalent to the following question: 

\begin{question} \label{Q1.4}
Whether 
$\{\|z^n\|^2_{\mathscr H_{\kappa_1\kappa_2}}\}_{n \in \mathbb N}$ is a Hausdorff moment sequence for Hausdorff moment sequences $\{\|z^n\|^2_{\mathscr H_{\kappa_1}}\}_{n \in \mathbb N}$ and $\{\|z^n\|^2_{\mathscr H_{\kappa_2}}\}_{n \in \mathbb N}$ ?
\end{question}

The same question in the context of hyponormal Hilbert modules was settled in the affirmative in \cite{BS} around the same time (other variations of Question \ref{Q1.4} have also been investigated, for example, see \cite[Proposition 6]{At}, \cite[Theorem 1.1]{BD}, and \cite[Theorem 7.1]{CLY}). For subnormal Hilbert modules, it was believed that the answer is no. Indeed, this is true as shown
in the following theorem.
\begin{theorem} \label{mainthm}
Let $s_1$ and $s_2$ be positive real numbers and
let $\mathscr H_{\kappa}$ denote the module tensor product 
$L^2_a(\mathbb D, w_{s_1}) \otimes_{\mathbb C[z]} L^2_a(\mathbb D, w_{s_2})$
of the weighted Bergman Hilbert modules $L^2_a(\mathbb D, w_{s_1})$ and $L^2_a(\mathbb D, w_{s_2})$. 
Then $1/{\|z^n\|^{2}_{\mathscr H_{\kappa}}}$
is a degree $3$ polynomial in $n \in \mathbb N,$ say, $p$. If $\tilde{p}$ denotes the analytic extension of $p$ to the complex plane,
then we have the following:
\begin{enumerate}
\item if $\tilde{p}$ has all real roots  then $\mathscr H_{\kappa}$ is a subnormal Hilbert module if and only if all roots of $\tilde{p}$ lie in $\mathbb L_0$;
\item if $\tilde{p}$ has a non-real complex root $z_0$ then $\mathscr H_{\kappa}$ is a subnormal Hilbert module if and only if $z_0$ lies in the closure of ${\mathbb L}_{-1}$.
\end{enumerate}
where, for a real number $r,$ $\mathbb L_r$ denotes the open left half plane
\beqn
\mathbb L_r :=\{z \in \mathbb C : ~\mbox{real part of}~z~\mbox{is less than~}r\}.
\eeqn
\end{theorem}

\begin{corollary} \label{coro}
Let $s_1$ and $s_2$ be positive real numbers and
let $\mathscr H_{\kappa}$ denote the module tensor product 
$L^2_a(\mathbb D, w_{s_1}) \otimes_{\mathbb C[z]} L^2_a(\mathbb D, w_{s_2})$
of the weighted Bergman Hilbert modules $L^2_a(\mathbb D, w_{s_1})$ and $L^2_a(\mathbb D, w_{s_2})$. Let $\mathfrak s$ and $\mathfrak p$ denote the sum and product of $s_1, s_2$ respectively.
Then we have the following:
\begin{enumerate}
\item If $({3\mathfrak s - \mathfrak p})^2 \geq 24\, {\mathfrak p}$ then $\mathscr H_{\kappa}$ is subnormal if and only if $3\mathfrak s > \mathfrak p$.
\item If $({3\mathfrak s - \mathfrak p})^2 < 24\, {\mathfrak p}$ then $\mathscr H_{\kappa}$ is a subnormal Hilbert module if and only if 
$\mathfrak s \geq \mathfrak p.$
\end{enumerate}
\end{corollary}
\begin{remark}
If $\mathfrak s \geq \mathfrak p$ then $L^2_a(\mathbb D, w_{s_1}) \otimes_{\mathbb C[z]} L^2_a(\mathbb D, w_{s_2})$ is always subnormal. 
Also, if $3 \mathfrak s \leq \mathfrak p$ then $L^2_a(\mathbb D, w_{s_1}) \otimes_{\mathbb C[z]} L^2_a(\mathbb D, w_{s_2})$ is never subnormal. 
\end{remark}


The following is immediate  from the preceding remark.
\begin{corollary}
For a positive number $s,$ let $\mathscr H_{\kappa}$ denote the module tensor product 
$L^2_a(\mathbb D, w_{s}) \otimes_{\mathbb C[z]} L^2_a(\mathbb D, w_{s})$
of the weighted Bergman Hilbert module $L^2_a(\mathbb D, w_{s})$ with itself.
Then we have the following:
\begin{enumerate}
\item If $s \leq 2$ then $\mathscr H_{\kappa}$ is always subnormal.
\item If $s \geq 6$ then $\mathscr H_{\kappa}$ never subnormal.
\end{enumerate}
\end{corollary}

The proofs of Theorem \ref{mainthm} and Corollary \ref{coro} will be presented in the next section.
Let us illustrate these results with the help of some instructive examples of different flavor. 

\begin{figure}[H]
  \centering
  \begin{subfigure}[$\left\{(s_1, s_2): ({3\mathfrak s - \mathfrak p})^2 \geq 24\, {\mathfrak p}~\mbox{and}~3\mathfrak s > \mathfrak p\right\}$]
    {\includegraphics[width = 0.49\textwidth]{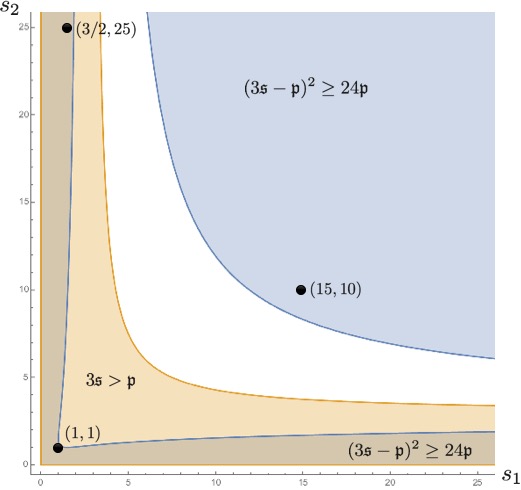}}
  \end{subfigure}%
  ~
  \begin{subfigure}[$\left\{(s_1, s_2) : ({3\mathfrak s - \mathfrak p})^2 < 24\, {\mathfrak p}~\mbox{and}~\mathfrak s \geq \mathfrak p \right\}$]
    {\includegraphics[width = 0.49\textwidth]{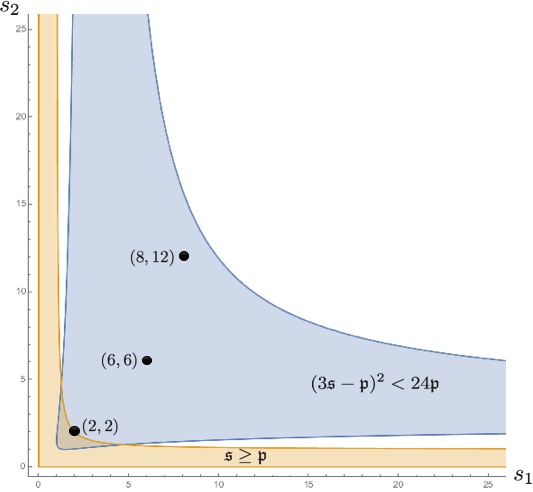}}
  \end{subfigure}%
  \caption{
Regions of subnormality of the module tensor product $L^2_a(\mathbb D, w_{s_1}) \otimes_{\mathbb C[z]} L^2_a(\mathbb D, w_{s_2})$ in $s_1$-$s_2$ plane, where $\mathfrak s$ and $\mathfrak p$ denote the sum
and product of $s_1, s_2$ respectively;  specific cases
investigated in Example \ref{1.4} 
have been shown as points in $s_1$-$s_2$ plane.  
  }
  \label{fig:fg_fg_disc_hc}
\end{figure}

\begin{example} \label{1.4}
Let $\mathscr H_{\kappa}$, $\tilde{p}$ be as given in Theorem \ref{mainthm} and let 
\beqn
\mathcal D_{m} = \sum_{j = 0}^m (-1)^j  \binom{m}{j} \| z^{j}
\|^2_{\mathscr{H}_{\kappa}}~(m \in \mathbb N). 
\eeqn
Let $\mathfrak s$ and $\mathfrak p$ denote the sum and product of $s_1, s_2$ respectively.
Then
we have the following:
\vskip.2cm
\begin{enumerate}
\item The case in which $\tilde{p}$ has real roots, that is, $({3\mathfrak s - \mathfrak p})^2 > 24\, {\mathfrak p}$
(see Figure 1(a)):
\begin{enumerate}
\item
If  $s_1 = 15, s_2 = 10$ then  $3\mathfrak s < \mathfrak p.$
In this case, the roots of $\tilde{p}$ are $-1, 1/10, 2/5$ and $\mathcal D_{75} < 0.$ 

\item If $s_1=1, s_2=1$ then $\mathfrak s > \mathfrak p.$
In this case, the roots of $\tilde{p}$ are $-1, -2, -3$ and
$\mathcal D_{m} \geq 0$ for all $m \in \mathbb N.$ The later part may be concluded from Corollary \ref{coro}(1).

\item If $s_1 =3/2 , s_2 =25$ then $3\mathfrak s > \mathfrak p.$
In this case, roots are $-1, 2(-7 \pm 2\sqrt{6})/25$ and $\mathcal D_m \ge 0$ for all $m \in \mathbb N$. Once again, the later part may be concluded from Corollary \ref{coro}(1).
\end{enumerate}

\vskip.2cm

\item The case in which $\tilde{p}$ has a complex root, that is,  $({3\mathfrak s - \mathfrak p})^2 < 24\, {\mathfrak p}$ (see Figure 1(b)):
\begin{enumerate}
\item 
If $s_1 = 6, s_2 = 6,$ then $\mathfrak s < \mathfrak p.$
In this case, the  
roots of $\tilde{p}$ are $-1, \pm i/\sqrt{6}$ and
$\mathcal D_{73}  < 0.$

\item
If
$s_1 = 8, s_2 = 12,$ then $\mathfrak s < \mathfrak p.$
In this case, the 
roots of $\tilde{p}$ are $-1, (3 \pm i\sqrt{7})/16$ and
$\mathcal D_{73}  < 0.$
\item If $s_1=2, s_2=2$ then $\mathfrak s = \mathfrak p.$
In this case, the roots of $\tilde{p}$ are $-1, -1 \pm i/\sqrt{2}$ and
$\mathcal D_{m} \geq 0$ for all $m \in \mathbb N.$
The later part may be concluded from Corollary \ref{coro}(2).
\end{enumerate}
\end{enumerate}
\end{example}


In parts (1)(a) and (2)(a) of  Example \ref{1.4},  for sufficiently large values of $m$, the multiplication operator $\mathscr M_z$ on $\mathscr H_{\kappa}$ is not $m$-hypercontractive in the sense of J. Agler \cite{Ag}. 
Moreover, in these examples, $m=75, 74$ are the smallest integers for which $m$-hypercontractivity of $\mathscr M_z$ fails. It also indicates that
direct verification of the non-subnormality of $\mathscr M_z$ is tedious.

\section{Proof of Theorem \ref{mainthm}}

Recall that, for a real number $r,$ $\mathbb L_r$ denotes the open left half plane
\beqn
\mathbb L_r =\{z \in \mathbb C : ~\mbox{real part of}~z~\mbox{is less than~}r\}.
\eeqn

In the proof of Theorem \ref{mainthm}, we need the following lemma. 
\begin{lemma}
If $a_0 \in (-\infty, 0)$ and $a_1, a_2 \in \mathbb L_0,$  then
\beqn
\frac{1}{(n-a_0)(n-a_1)(n-a_2)} = \int_{[0, 1]}t^n w(t) \, dt~(n \in \mathbb N),
\eeqn
where the weight function $w : [0, 1] \rar (0, \infty)$ is given as follows:
\begin{enumerate}   
\item If $a_0, a_1, a_2$ are distinct real numbers, then
\beqn
w(t)= \frac{t^{-a_0-1}}{(a_0-a_1)(a_0-a_2)}  + \frac{t^{-a_1-1}}{(a_1-a_0)(a_1-a_2)} + \frac{t^{-a_2-1}}{(a_2-a_0)(a_2-a_1)}~(t \neq 0).
\eeqn
\item If $a_1 =a_0$ and $a_2$ is a real number not equal to $a_0$, then
\beqn
w(t)=  \frac{1}{(a_0-a_2)^2}\, (t^{-a_2-1} - t^{-a_0-1}) - \frac{1}{a_0-a_2} \, t^{-a_0-1} \, \log t~(t \neq 0). 
\eeqn
\item If $a_0=a_1=a_2$, then
\beqn
w(t)=\frac{1}{2} \, t^{-a_0-1} (\log t)^2~(t \neq 0).
\eeqn
\item If $a_1=a+ib, ~a_2=a-ib$ are complex numbers with $b >0$ then
\beqn
w(t)= \frac{1}{(a_0-a)^2 + b^2}\Big(t^{-a_0-1} - \frac{\sqrt{(a_0-a)^2+b^2}}{b} \, t^{-a-1} \sin (b \log t + \theta) \Big)~(t \neq 0),
\eeqn
where $\theta$ denotes the principal argument of $a_1-a_0$.
\end{enumerate}
\end{lemma}
\begin{proof}
The first three parts are special cases of \cite[Theorem 3.1]{AC}. 
To see the last part, note that
\beqn
&& \frac{1}{(n-a_0)(n-a_1)(n-a_2)}  \\ =  & \frac{1}{(a-a_0)^2 + b^2} \Big(\frac{1}{n-a_0} & + \frac{a_2-a_0}{2ib}\frac{1}{n-a_1} -  \frac{a_1-a_0}{2ib}\frac{1}{n-a_2} \Big).
\eeqn
Further, the term in brackets can be rewritten as
\beqn
\int_{0}^1 t^{n-a_0-1}dt + \frac{a_2-a_0}{2ib}\int_{0}^1 t^{n-a_1-1}dt -  \frac{a_1-a_0}{2ib}\int_{0}^1 t^{n-a_2-1}dt.
\eeqn
However, \beqn \frac{a_2-a_0}{2ib} t^{-a_1-1} -  \frac{a_1-a_0}{2ib} t^{-a_2-1}=\\
-\frac{\sqrt{(a-a_0)^2+ b^2}}{b}  \Big(\frac{e^{i(b \log t + \theta)} -e^{-i(b \log t + \theta)} }{2i}\Big) t^{-a-1}, \eeqn
where $\theta$ denotes the principal argument of $a_1-a_0$. It is now easy to see that $w(t)$ has the desired form.
\end{proof}
\begin{remark}
Note that in all the above cases, $\int_{[0, 1]}w(t)dt \in (0, \infty).$ 
\end{remark}

\begin{proof}[Proof of Theorem \ref{mainthm}] We have already recorded in \eqref{moments} that $$\|z^n\|^2_{L^2_a(\mathbb D, w_{s_j})}=\frac{1}{s_j n + 1}~\mbox{for~} n \in \mathbb N~\mbox{and} ~j=1, 2.$$
It now follows from \eqref{relation-m-s} that
\beqn
\|z^n\|^{2}_{\mathscr H_{\kappa}} =\frac{1}{\displaystyle \sum_{k=0}^n (s_1 k + 1)(s_2(n-k)+1)}~(n \in \mathbb N).
\eeqn
However, \beq \label{positive}
\sum_{k=0}^n (s_1 k + 1)(s_2(n-k)+1) = \frac{1}{6}(n+1)( s_1s_2n^2 + (3(s_1 + s_2) - s_1s_2)n + 6) \notag \\
= \frac{1}{\alpha_1 \alpha_2}(n+1)(n-\alpha_1)(n-\alpha_2), \eeq 
where $\alpha_1, \alpha_2$ are given by
\beq \label{roots}
\alpha_1:=\frac{-\gamma + \sqrt{\gamma^2 - 24 s_1 s_2}}{2s_1s_2}, \quad \alpha_2:=\frac{-\gamma - \sqrt{\gamma^2 - 24 s_1 s_2}}{2s_1s_2}
\eeq
with $\gamma:=3(s_1 + s_2) - s_1s_2.$
It follows that 
$$\|z^n\|^{2}_{\mathscr H_{\kappa}}=\frac{{\alpha_1 \alpha_2}}{(n+1)(n-\alpha_1)(n-\alpha_2)}~(n \in \mathbb N).$$ 
This shows that
$1/{\|z^n\|^{2}_{\mathscr H_{\kappa}}}$
is a degree $3$ polynomial $p$ in $n \in \mathbb N.$ 

Let $\tilde{p}$ denote the analytic extension of $p$ to the complex plane.
Assume now that one of $\alpha_1$ and $\alpha_2$ belongs to $\mathbb C \setminus \mathbb L_0.$ By \eqref{positive} and \eqref{roots}, however, both $\alpha_1$ and $\alpha_2$ must belong to  $\mathbb C \setminus \mathbb L_0.$
We contend that the sequence 
\beqn
\Big\{\|z^n\|^{2}_{\mathscr H_{\kappa}}=\frac{{\alpha_1 \alpha_2}}{(n+1)(n-\alpha_1)(n-\alpha_2)}\Big\}_{n \in \mathbb N}
\eeqn
is not a Hausdorff moment sequence. On the contrary, suppose that there exists a probability measure $\mu$ on $[0, 1]$ such that \beqn \|z^n\|^{2}_{\mathscr H_{\kappa}} =\int_{[0, 1]} t^n d\mu~({n \in \mathbb N}). \eeqn 
Let $n_0$ be the smallest integer bigger than the maximum of real parts of $\alpha_1$ and $\alpha_2.$  
Consider now the completely monotone sequence $\{\|z^{n+n_0}\|^{2}_{\mathscr H_{\kappa}}\}_{n \in \mathbb N}$ (Remark \ref{translate}),
and note that by the preceding lemma,
\beqn
\|z^{n+n_0}\|^{2}_{\mathscr H_{\kappa}} = \frac{{\alpha_1 \alpha_2}}{(n + n_0+1)(n + n_0-\alpha_1)(n + n_0-\alpha_2)} =\int_{[0, 1]}t^n w(t)dt
\eeqn
for some non-negative integrable function $w : [0, 1] \rar (0, \infty).$
On the other hand, $\|z^{n+n_0}\|^{2}_{\mathscr H_{\kappa}}=\int_{[0, 1]} t^n t^{n_0} d\mu(t).$ By the determinacy of the Hausdorff moment problem \cite{BCR}, we must have
\beqn
t^{n_0} d\mu(t) =w(t)dt.
\eeqn
We now apply the previous lemma to $a_0:=-n_0-1$, $a_1:=\alpha_1-n_0$ and $a_2:=\alpha_2-n_0$ to conclude that $t^{-n_0}w(t)$ takes one of the following expressions (for some scalars $c_0, c_1, c_2, \theta$): 
\begin{enumerate}   
\item If $a_1, a_2$ are distinct real numbers, then
\beqn
t^{-n_0}w(t)= c_0   + c_1 t^{-\alpha_1-1} + c_2 t^{-\alpha_2-1}.
\eeqn
\item If $a_1 =a_2$, then
\beqn
t^{-n_0} w(t)=  c_0 (1 - t^{-\alpha_1-1}) + c_1 t^{-\alpha_1-1} \log t. 
\eeqn
\item If $a_1=c+ib, ~a_2=c-ib$ are complex numbers with $b >0$ then
\beqn
t^{-n_0} w(t)= c_0 t^{-a_0-1}+ c_1 t^{-\alpha-1} \sin\Big(b \log t + \theta\Big),
\eeqn
where $\theta$ is the principal argument of $a_1-a_0$.
\end{enumerate} 
One way to get a contradiction is to show that the integral of $t^{-n_0}w(t)$ over $[0,1]$ is not convergent. However, to avoid computations, one can alternatively arrive at the contradiction through an interpolation result on completely monotone sequences \cite{W}. Toward this, observe that $\{\|z^n\|^{2}_{\mathscr H_{\kappa}}\}_{n \in \mathbb N}$ is minimal in the sense that
$\mu(\{0\})=0,$ where we used the convention that 
$\infty \cdot 0 = 0.$  
By \cite[Proposition 4.1]{AC}, all the zeros of  $\tilde{p}(z)={(z+1)(z-\alpha_1)(z-\alpha_2)}$ must lie in $\mathbb L_0$, which contradicts the assumption that at least one of  $\alpha_1$ and $\alpha_2$ belongs to 
$\mathbb C \setminus \mathbb L_0.$ 
The remaining part in (1) is immediate from \cite[Theorem 3.1]{AC}.

To see (2), assume that $\tilde{p}$ has a non-real complex root $\alpha_1=a+ib \in \mathbb L_0$ with $b > 0$.  In view of the preceding discussion, it suffices to show 
that $\mathscr H_{\kappa}$ is a subnormal Hilbert module if and only if $\alpha_1$ lies in the closure of ${\mathbb L}_{-1}$.
By (4) of the preceding lemma, 
\beq \label{w(t)-proof}
w(t) &=& \frac{\alpha_1\alpha_2}{(1+a)^2 + b^2}\Big(1 - \frac{\sqrt{(1+a)^2+b^2}}{b} \, t^{-a-1} \sin (b \log t + \theta) \Big) \notag \\
&=& \frac{|\alpha_1|^2}{(1+a)^2 + b^2}\Big(t^{a+1} - \frac{{\sin (b \log t + \theta)}}{\sin \theta} \Big)t^{-a-1},
\eeq
where $\theta$ denotes the principal argument of $a+1+ib$.
If $a=-1$ then $w(t) = \frac{|\alpha_1|^2}{b^2} (1 - \sin(b \log t + \theta )) \geq 0$ for $t \in (0, 1]$, and hence $\|z^{n}\|^{2}_{\mathscr H_{\kappa}}$ is a Hausdorff moment sequence.
Assume now that $\alpha_1=a+ib$ lies in ${\mathbb L}_{-1}$, that is, $a < -1.$ Note that $\pi/2 < \theta < \pi.$
Consider the function
\beqn
F(t) := t^{a+1} - \frac{{\sin (b \log t + \theta)}}{\sin \theta}~(t > 0),
\eeqn
and note that $F'(t) \leq \frac{1}{t} \Big( a+1 - \frac{b \cos(b \log t + \theta)}{\sin \theta} \Big).$ Let $t_0:= e^{-2 \theta/b}$. Then, for $t \in [t_0, 1],$ we have
\beqn
F'(t) \leq \frac{1}{t} \Big( a+1 - \frac{b \cos \theta }{\sin \theta} \Big)=0.
\eeqn
Hence $F(t) \geq F(1)=0$ for all $t \in [t_0, 1].$ For $t \in (0, t_0),$ we have
\beqn
F(t) & \geq & e^{-2\theta(a+1)/b} - \frac{1}{\sin \theta} \\ & \geq & 
e^{-2(a+1)/b} - \sqrt{1 + \Big(\frac{a+1}{b}\Big)^2} \\
& \geq & e^{-2(a+1)/b} - \left({1 + \frac{1}{2}\Big(\frac{a+1}{b}\Big)^2}\right), 
\eeqn
which is clearly positive. It is now clear from \eqref{w(t)-proof} that $\|z^{n}\|^{2}_{\mathscr H_{\kappa}}$ is a Hausdorff moment sequence.

Assume next that $-1 < a < 0.$ Then $0 < \theta < \pi/2.$
Choose $m \in \mathbb Z$ such that $t^{a+1}_m < \frac{1}{2\sin \theta},$ where $t_m := e^{({2m \pi + \pi/2 - \theta})/{b}}.$
Then, by \eqref{w(t)-proof}, 
\beqn
w(t_m) =
&=& \frac{|\alpha_1|^2}{(1+a)^2 + b^2}\Big(t^{a+1}_m - \frac{1}{\sin \theta} \Big)t^{-a-1}_m \\
& \leq & 
-\frac{|\alpha_1|^2}{(1+a)^2 + b^2}\frac{t^{-a-1}}{2\sin \theta} < 0.
\eeqn
By the continuity of $w$, $w < 0$ in a neighborhood of $t_m.$ It follows that $\|z^{n}\|^{2}_{\mathscr H_{\kappa}}$ is not a Hausdorff moment sequence.
This completes the proof.
\end{proof}

Finally, we note that the conclusions in Corollary \ref{coro} follow from Theorem \ref{mainthm} and \eqref{roots}.

\medskip \textit{Acknowledgment}. \
The authors gratefully acknowledge Gadadhar Misra for drawing their attention to Salinas's question
that had remained unresolved thus far. They thank him for his constant encouragement
throughout the preparation of this manuscript.

\end{document}